\documentclass[10pt]{amsart}%
\usepackage{amsthm}
\usepackage{amsmath}
\usepackage{amssymb}
\usepackage{latexsym}
\usepackage{enumerate}
\usepackage{graphicx}
\usepackage[utf8]{inputenc}
\usepackage{epsfig}
\usepackage{pgf}
\usepackage{tikz}
\usepackage{float}
\usepackage{dsfont}
\usepackage{times}
\usepackage{amsfonts}
\providecommand{\U}[1]{\protect\rule{.1in}{.1in}}

\newtheorem{theorem}{Theorem}[section]
\newtheorem{lemma}[theorem]{Lemma}

\newtheorem{proposition}[theorem]{Proposition}
\newtheorem{corollary}[theorem]{Corollary}

\theoremstyle{definition}
\newtheorem{definition}[theorem]{Definition}

\newtheorem{remark}[theorem]{Remark}

\numberwithin{equation}{section}

\newcommand{\Vol}{\operatorname{Vol}}

\newcommand{\erre}{\mathbb{R}}

\newcommand{\Han}{\mathbb{H}^n(b)}
\newcommand{\Hatwo}{\mathbb{H}^2(b)}

\newcommand{\N}{\mathbb{N}}

\newcommand{\Sup}{\operatorname{Sup}}
\newcommand{\Inf}{\operatorname{Inf}}
\begin{document}
\title[Mean curvature and compactification]{Mean curvature and compactification of surfaces in a negatively curved Cartan-Hadamard manifold}
\author[A. Esteve]{Antonio Esteve*}
\address{I.E.S. Alfonso VIII, Cuenca - Departamento de An\'alisis Econ\' omico y Finanzas, Universidad de Castilla la Mancha, Cuenca, Spain.}
\email{aesteve7@gmail.com}
\author[V. Palmer]{Vicente Palmer**}
\address{Departament de Matem\`{a}tiques - Institute of New Imaging Technologies,
Universitat Jaume I, Castellon, Spain.}
\email{palmer@mat.uji.es}
\thanks{}
\thanks{* Work partially supported by DGI grant MTM2010-21206-C02-02.}
\thanks{* Work partially supported by the Caixa Castell\'{o} Foundation, and DGI grant MTM2010-21206-C02-02.}
\maketitle

\begin{abstract}
 We state and prove a Chern-Osserman-type inequality in terms of
the volume growth for complete surfaces with controlled mean curvature properly immersed in a
Cartan-Hadamard manifold $N$ with sectional curvatures bounded from above by a
negative quantity $K_{N}\leq b<0$.

\end{abstract}


\section{Introduction}

\label{secIntro}

\bigskip In the articles \cite{Ch2} and \cite{Che3}, Chen Qing and Cheng Yi proved the finiteness of the topology and the following
Chern-Osserman-type inequality for complete and properly immersed minimal surfaces in $\Han$ with finite total extrinsic curvature $\int_S\Vert A^S\Vert^2 d\sigma<\infty$ (here $\Vert A^S\Vert $ denotes the Hilbert-Schmidt norm of the second fundamental form of $S$ in $\Han$): 
\begin{equation}\label{choss1}
-\chi(S)\leq\frac{1}{4\pi}\int_{S}\Vert A^{S}\Vert^{2}d\sigma
-\operatorname{Sup}_{r}\frac{\operatorname{Vol}(S^{2}\cap B_{r}^{b,n}
)}{\operatorname{Vol}(B_{r}^{b,2})}.
\end{equation}
where $\chi(S)$ is the Euler characteristic of the surface, $B_{r}^{b,n}$ denotes the geodesic $r$-ball in $\Han$ and $\frac{\operatorname{Vol}(S^{2}\cap B_{r}^{b,n}
)}{\operatorname{Vol}(B_{r}^{0,2})}$ is the volume growth of the domains $S^{2}\cap B_{r}^{b,n}$.

A natural question arises in this context: can we prove the finiteness of the topology of a not necessarily minimal surface in a Cartan-Hadamard manifold and, moreover, establish a Chern-Osserman-type inequality for its Euler characteristic? (At this point we are referring to the work \cite{W}, where the finiteness of the topology and a Chern-Osserman inequality are proven for not necessarily minimal surfaces in the Euclidean spaces $\erre^n$).

In this paper we provide a partial answer to this question. We consider a complete and connected surface $S$ properly immersed in a Cartan-Hadamard manifold $N$ with sectional curvatures $K_N$ bounded from above by $b <0$. As in \cite{Che3}, we assume that $\int_S\Vert A^S\Vert^2 d\sigma<\infty$ and that the sectional curvatures of the ambient manifold $N$ satisfy $\int_{S}(b-K_N)d\sigma <\infty$. On the other hand, we assume that the mean curvature of $S$ in $N$, $H_S$, is controlled by a radial function $h(r)$ (which depends on the distance $R$ to a fixed pole $o \in N$) and its total mean curvature
 $\int_S\Vert H_S \Vert d\sigma$ is finite. Then we obtain a Chern-Osserman-type inequality, thereby proving that the topology of such non-minimal surfaces is finite and generalizing the results directly in \cite{Ch2} and \cite{Che3}.

The monotonicity and finiteness of the volume growth function $\frac{\operatorname{Vol}(D_{t})}{\cosh\sqrt{-b}t}$ (or a modified version of it) associated to the distinguished domains $D_t \subset S$ called {\em extrinsic balls} (see Definition \ref{ExtBall}) plays a fundamental r\^ole in the description of the topology of the surface. This monotonicity property is obtained from some isoperimetric inequalities satisfied by the extrinsic balls in $S$.

The isoperimetric inequalities are based, in turn, on the application of the divergence theorem and comparison with the Laplacian operator acting on radial functions which comes from the Hessian-Index analysis for manifolds with a pole that we can find in \cite{GW}, (see also \cite{MM} and \cite{Pa3}).

We basically follow the arguments set out in the works \cite{Ch2} and \cite{Che3}. However, several analytical and topological difficulties arising from the fact of considering an ambient space with variable curvature had to be overcome.

In particular, we present the following estimation of the Euler characteristic of an immersed surface
$$-\chi(S) \leq \lim_{t \to \infty} (-\chi(D_t))$$
for an accurate exhaustion of $S$ by {\em connected} extrinsic balls $\{D_t\}_{t>0}$, (see Theorem \ref{Huber} in section \S.4). The proof of this result is based on the proof of Huber's classical theorem given by White in \cite{W}. This is a key result which will allow us to argue in a similar way as in \cite{Ch2} and \cite{Che3}, even though our ambient manifold has no constant curvature. 


\subsection{Outline of the paper}

In section \S .2 we present the basic tools we use (such as the co-area formula) and the definitions and facts about the rotationally
symmetric spaces used as a model for comparison purposes. In Section \S.3 we state our main results: Theorem \ref{th_main_nomin1} and Corollary \ref{Cor1}, and prove Corollary \ref{Cor1}. Section \S.4 is divided into two parts: Subsection \S.4.1 is devoted to the basic results about the
Hessian comparison theory of restricted distance function that we are going to use
(see proposition \ref{corLapComp}) and in Subsection \S.4.2 an estimate of the geodesic curvature of the boundary of the extrinsic balls covering the surface is obtained as a corollary, using the Gauss-Bonnet theorem, and from there an estimation of the Euler characteristic of such extrinsic balls is also calculated.
Section \S .5 is divided into two parts: in Subsection \S.5.1 the monotonicity property of the volume growth is studied in the non-minimal context, and in Subsection \S.5.2 the estimation of the Euler characteristic of the surface is presented in terms of the Euler characteristics of the extrinsic balls. Section
\S.6 is devoted to the proof of the main result, (Theorem
\ref{th_main_nomin1}).

\medskip
\section{Preliminaries}
\medskip
We are now going to present the precise \emph{controlled mean curvature} setting,
where we can prove Chern-Osserman-type inequality by introducing the notion of
\emph{comparison constellations}.

We assume throughout the paper that $\varphi: S \longrightarrow N$ is a complete and proper immersion of a non-compact surface $S$ in a Cartan-Hadamard manifold $N$. Throughout the paper, we identify $\varphi(S)\equiv S$ and $\varphi(x)\equiv x$ for all $ x \in S$. We also assume that the Cartan-Hadamard manifold $N^{n}$ has 
sectional curvatures bounded from above by a negative bound $K_N \leq b <0$. All the points in these manifolds are poles. Recall that a
pole is a point $o$ such that the exponential map $\exp_{o}\colon T_{o}
N^{n}\rightarrow N^{n}$ is a diffeomorphism. For every $x\in N^{n}
\setminus\{o\}$ we define $r(x)=\operatorname{dist}_{N}(o,x)$, and this
distance is realized by the length of a unique geodesic from $o$ to $x$, which
is the \textit{radial geodesic from $o$}. We also denote by $r$ the
restriction $r|_{S} :S\rightarrow\mathbb{R}_{+}\cup\{0\}$. This restriction is
called the \emph{extrinsic distance function} from $o$ in $S^{m}$. The
gradients of $r$ in $N$ and $S$ are denoted by $\nabla^{N}r$ and $\nabla^{S}
r$, respectively. Let us remark that $\nabla^{S}r(x)$ is just the tangential
component in $S$ of $\nabla^{N}r(x)$, for all $x\in S$. Then we have the
following basic relation:
\begin{equation}
\nabla^{N}r=\nabla^{S}r+(\nabla^{N}r)^{\bot}, \label{radiality}
\end{equation}
where $(\nabla^{N}r)^{\bot}(x)=\nabla^{\bot}r(x)$ is perpendicular to $T_{x}S$
for all $x\in S$.

We are now going to define the {\em extrinsic balls}:

\begin{definition}
\label{ExtBall} Given a connected and complete surface $S^{2}$ in a
Cartan-Hadamard manifold $N^{n}$, we denote the \emph{extrinsic metric balls}
of radius $R$ and center $o\in N$ by $D_{R}(o)$. They are defined as any connected component of the intersection:
\[
B_{R}(o)\cap S=\{x\in S\colon r(x)<R\},
\]
where $B_{R}(o)$ denotes the open geodesic ball of radius $R$ centered at the
pole $o$ in $N^{n}$.
\end{definition}

\begin{remark}
\label{theRemk0} The restriction $r|_{S}$ is smooth in $S$ and consequently, by Sard's theorem and the Regular Level Set Theorem, 
the radii $R$ that produce smooth boundaries $\partial D_{R}(o)$ are dense in
$\mathbb{R}$.
\end{remark}

\begin{definition}
\label{defRadCurv} Let $o$ be a pole in the ambient Cartan-Hadamard manifold $N$ and let
$x\in M\setminus\{o\}$. The sectional curvature $K_{N}(\sigma_{x})$ of the
two-plane $\sigma_{x}\in T_{x}N$ is then called an \textit{$o$-radial
sectional curvature} of $N$ at $x$ if $\sigma_{x}$ contains the tangent vector
to a minimal geodesic from $o$ to $x$. We denote these curvatures by
$K_{o,N}(\sigma_{x})$.
\end{definition}

At this point we should remark that we assume that the $o$-radial sectional curvatures of $N$ are
bounded from above by the constant $b<0$, which is the constant sectional curvature of the Hyperbolic space $\Han$. This space can be viewed as a special kind of rotationally symmetric space called a {\em model} space.

\begin{definition}
[see \cite{Gri}, \cite{GW}]\label{defModel} A $\omega-$model $M_{\omega}^{m}$
is a smooth warped product with base $B^{1}=[0,\Lambda\lbrack\,\subset
\mathbb{R}$ (where $0<\Lambda\leq\infty$), fiber $F^{m-1}=\mathbb{S}_{1}
^{m-1}$ (i.e. the unit $(m-1)$-sphere with standard metric), and warping
function $\omega\colon\lbrack0,\Lambda\lbrack\rightarrow\mathbb{R}_{+}
\cup\{0\}$, with $\omega(0)=0$, $\omega^{\prime}(0)=1$, and $\omega(r)>0$ for
all $r>0$. The point $o_{\omega}=\pi^{-1}(0)$, where $\pi$ denotes the
projection onto $B^{1}$, is called the \emph{center point} of the model space.
If $\Lambda=\infty$, then $o_{\omega}$ is a pole of $M_{\omega}^{m}$.
\end{definition}

\begin{remark}
\label{propSpaceForm} The Hyperbolic space $\Han$ is a w$-$model with warping function
$$\omega_b(r)=\frac{1}{\sqrt{-b}}\sinh(\sqrt{-b}\,r)$$
\end{remark}

\begin{proposition}
[see \cite{O'N}, \cite{GW} and \cite{Gri}]\label{propWarpMean} Let $M_{\omega
}^{m}$ be a $\omega-$model with warping function $\omega(r)$ and center
$o_{\omega}$. The distance sphere of radius $r$ and center $o_{\omega}$ in
$M_{\omega}^{m}$ is the fiber $\pi^{-1}(r)$. This distance sphere has the
constant mean curvature $\eta_{\omega}(r)=\frac{\omega^{\prime}(r)}
{\omega(r)}$. On the other hand, the $o_{\omega}$-radial sectional curvatures
of $M_{\omega}^{m}$ at every $x\in\pi^{-1}(r)$ (for $r>0$) are all identical
and determined by $K_{o_{\omega},M_{\omega}}(\sigma_{x})=-\frac{\omega
^{\prime\prime}(r)}{\omega(r)}.$
\end{proposition}
\begin{remark}
\label{theRemk1} The mean curvature of the geodesic spheres in the Hyperbolic space $\Han$,
'pointed inward', is (see \cite{Pa}):
\begin{equation*}
\eta_{\omega_{b}}(t)=\frac{\omega'_b(t)}{\omega_b(t)}=
\sqrt{-b}\coth\sqrt{-b}t
\end{equation*}

\end{remark}

\begin{definition}
[\cite{MP5}]\label{defCspace} Given a function $h: \erre \longrightarrow \erre$, the \emph{{isoperimetric comparison space}}
$C_{\omega_{b},h}^{\,2}\,$ is the $W-$model space $[\,0,+\infty\,)\times_W S^{0,1}_1$ with base interval
$B=[\,0,+\infty\,)$ and warping function $W(r)$ defined by:
\begin{equation}
W(t)=\frac{\sinh\sqrt{-b}t}{\sqrt{-b}e^{2\int_{0}^{t}h(s)ds}}.
\label{eqLambdaDiffeq}
\end{equation}
Note that
\begin{equation}
W(0)=0,W^{\prime}(0)=1\quad. \label{eqTR}
\end{equation}
\end{definition}
\begin{remark}
The warping function $W(r)$ is a/the solution of the differential 
equation:
\begin{equation} \label{eqLambdaDiffeq1}
\begin{aligned}
\frac{W^{\prime}}{W}=\frac{w_b^{\prime}}{w_b}-2h(t)
\quad.
\end{aligned}
\end{equation}
with the following boundary condition:
\begin{equation} \label{eqTR1}
W'(0) = 1
\quad .
\end{equation}
\end{remark}

\begin{remark}
We observe that $C^{\,2}_{w_b, h}\,$ is indeed a model space $C^{\,2}_{w_b, h}\,=M^2_W$ with a
well-defined pole $o_{W}$ at $r = 0$: $W(r) \geq 0$ for all $r$
and $W(r)$ is only $0$\, at \,$r=0$, where, additionally, because of
 equation (\ref{eqTR1}), $W'(0)\, =
\, 1\,$.

On the other hand, given the warping function $\,W(r)\,$, we
introduce the isoperimetric quotient function
$\,q_{W}(r)\,$ for the corresponding $W-$model
space $C^{\,2}_{w_b, h}\,$ as follows:
\begin{equation} \label{eqDefq}
q_{W}(r) \, = \,
\frac{\Vol(B_{r}^{W})}{\Vol(S_{r}^{W})} \, = \,
\frac{\int_{0}^{r}\,W(t)\,dt}{W(r)}
\quad ,
\end{equation}
where $\,B_{r}^{W}\,$ denotes the polar centered
geodesic $r-$ball of radius $\,r\,$ in
$C^{\,2}_{w_b, h}\,$ with boundary sphere
$\,S_{r}^{W}\,$.

\end{remark}

\section{Main Results}\label{mainresults}

With these definitions in hand, we present the notion of \textit{strongly balanced isoperimetric
comparison space} and our main results:


\begin{definition}
\label{constelacion_adaptada} 
The isoperimetric comparison space $C_{\omega_{b},h}^{2}$ is \emph{strongly balanced} on the interval $[0,\infty)$ if and only if the following inequality holds for all $r \geq 0$

\begin{equation}\label{eqBalB}
\vert h(r)\vert \leq \frac{1}{2} (\eta_{w_b}(r)-\sqrt{-b})
\end{equation}
where $\eta_{w_b}(r)=\sqrt{-b}\coth\sqrt{-b}r$ is the mean curvature of the geodesic $r$-spheres in the hyperbolic spaces $\Han$.
\end{definition}

With all these concepts and definitions in hand, we have our main result:

\begin{theorem}\label{th_main_nomin1}\bigskip\bigskip 
Let us consider $N^n$ to be a Cartan-Hadamard manifold, and $o \in N$ a pole in $N$. Let us suppose that its radial sectional curvatures are bounded from above by a negative bound $K_{o,N}(\sigma_x) \leq b <0$. Let $S$ be a complete, connected and properly immersed surface in $N$ such that there exists a radial function $h(r)$ satisfying
 \begin{equation}
C(x)= -\langle\nabla^{N}r(x), H_{S}(x) \rangle \leq h(r(x)) \quad{\text{for all}
}\quad x \in S \,\, 
\end{equation}
where $H_{S}(x)$ denotes the mean curvature vector of $S$ in $N$.

Let $C_{\omega_{b},h}^{2}$ denote the $W$-model constructed via $\omega_{b}$ and $h$, and assume that $C_{\omega_{b},h}^{2}$ is a strongly balanced isoperimetric
comparison space on the interval $[\,0,\infty\,)$.

Let us also assume that \newline
\begin{align*}
&  \text{(I) }\int_{S}(b-K_N\vert_S)d\sigma<+\infty\text{,}\\
&  \text{(II) }\int_{S}\Vert A^{S}\Vert^{2}d\sigma<+\infty\text{ }\\
&  \text{(III) }\int_{S}\Vert H_S\Vert d\sigma<+\infty\\
\end{align*}
where $A^{S}$ denotes the second fundamental form of $S$ in $N$, $K_N\vert_S$ denotes the sectional curvature of $N$ restricted to the tangent plane $T_pS$, for all $p \in S$, and $H_S$ denotes the mean curvature vector of $S$.

Then

\begin{enumerate}
\item $\operatorname{Sup}_{t>0}\frac{v(t)}{\cosh\sqrt{-b}t}<+\infty$, where
$v(t)=vol(D_{t})\,\,\forall t > 0$ and $D_{t}$ denotes the extrinsic ball on the surface $S$.

\item $S^{2}$ has finite topological type, and there exists $t_0 \geq 0$ such that
\begin{align}\label{mainIneq}
-\chi(S)  &  \leq\frac{1}{2\pi}\int_{S}(b-K_N\vert_S)d\sigma+\frac{1}{4\pi}\int
_{S}\Vert A^{S}\Vert^{2}d\sigma-\newline C\Sup_{t>0}\frac{vol(D_{t}
)}{\operatorname{Vol}(B_{t}^{b,2})}\\+
&  \frac{\sqrt{-b}}{\pi}\int_{S}\Vert H_S\Vert d\sigma
-\frac{bv(t_{0})}{2\pi}\nonumber
\end{align}
\end{enumerate}
\bigskip
where $C \in [0,1]$ is the constant defined as
$$C:=\underset{t>0}{\Inf}\left(  \cosh\sqrt{-b}t-q_{W}(t)\sqrt{-b}
\sinh\sqrt{-b}t\right)$$
\end{theorem}
\begin{remark}\label{remark3.4}
Given the surface $S^2$ immersed in $N$, the quantity $b-K_{N}\vert_S$, (where $K_{N}\vert_S=K_{N}(p,T_pS)$ is the sectional curvature of $N$ at $p\in S$ of the tangent plane $T_pS$), only depends on the points $p \in S$. Hence, the assumption $\int
_{S}(b-K_{N}\vert_S)d\sigma<+\infty$ makes sense. Indeed, when we consider the sectional curvature of the ambient manifold restricted to the $2$-dimensional tangent plane $T_pS \subseteq T_pN$, we have, by virtue of the Gauss formula and given an orthonormal basis $\{e_1,e_2\}$ of $T_pS$:
\begin{equation}\label{TheEquation}
\begin{aligned}
K_N(p, T_pS)& -K_S(p,T_pS)=\\&=\langle A^S(e_1,e_2),A^S(e_1,e_2)\rangle-\langle A^S(e_1,e_1),A^S(e_2,e_2)\rangle\\&=\frac{1}{2}\left(  \Vert A^{S}\Vert^{2}-4\Vert H_S\Vert
^{2}\right)
\end{aligned}
\end{equation}
\noindent where $K_S(p,T_pS)=K_S$ is the Gauss curvature of $S$. Hence, $b-K_{N}$ does not depend on the basis $\{e_1,e_2\}$ of $T_pS$. If there is no risk of confusion, we shall denote as $\int
_{S}(b-K_{N})d\sigma$ the integral $\int
_{S}(b-K_{N}\vert_S)d\sigma$.
\end{remark}

As a corollary of Theorem \ref{th_main_nomin1}, we have the following result, which is a generalization of the main theorem in \cite{Che3}, when we consider connected and minimal surfaces in a Cartan-Hadamard manifold, (see also \cite{EP}):

\begin{corollary}\label{Cor1}
Let $S^{2}$ be a complete, connected and minimal surface properly immersed in
a Cartan-Hadamard manifold $N$, with sectional curvatures bounded from above by a negative quantity $K_N \leq b <0$.

Let us suppose that 

\begin{equation}\label{Hyp1}
\int_{S}\Vert A^{S}\Vert ^{2}d\sigma<+\infty
\end{equation}
\noindent and 
\begin{equation}\label{Hyp2}
\int_{S}(b-K_{N}\vert_S)d\sigma<+\infty
\end{equation}
\noindent where $A^{S}$ denotes the second fundamental form of $S$ in $N$ and $K_N\vert_S$ denotes the sectional curvature of $N$ restricted to the tangent plane $T_pS$, for all $p \in S$.

Then:

\begin{enumerate}
\item $\Sup_{t>0}\frac{\operatorname{Vol}(D_{t})}{\operatorname{Vol}(B_{t}^{b,2}
)}<+\infty,$ 
\item $S^2$ has finite topological type,

\item $-\chi(S)\leq\frac{1}{4\pi}
\int_{S}\Vert A^{S}\Vert ^{2}d\sigma-\Sup_{t>0}\frac{\operatorname{Vol}
(D_{t})}{\operatorname{Vol}(B_{t}^{b,2})}+\frac{1}{2\pi}\int_{S}(b-K_{N})d\sigma. \\$
\end{enumerate}
where $D_t$ denotes the connected extrinsic ball on surface $S$.

\end{corollary}
\begin{proof}
As $S$ is minimal, then $\mathcal{C}(x)=0\,\,\,\forall x \in S$, so we consider $h(r)=0$ for all $r>0$. Then, $W(r)=w_b(r)$ trivially and we have that $C_{w_b,h}^{2}=\Hatwo$ is a strongly balanced isoperimetric comparison space. In particular, $t_0=0$.
 
 It is straightforward to check that $q_W(t)=q_{w_{b}}(t)=\frac{1}{\sqrt{-b}}(\coth(\sqrt{-b}t)-\frac{1}{\sinh(\sqrt{-b}t)})$, so $f_W(t)= \cosh\sqrt{-b}t-q_{W}(t)\sqrt{-b}
\sinh\sqrt{-b}t=1\,\,\forall t \geq 0$, and hence 
$$C=\underset{t>0}{\Inf}\left(  \cosh\sqrt{-b}t-q_{W}(t)\sqrt{-b}
\sinh\sqrt{-b}t\right)=1$$
\noindent  in this case.
\end{proof}

\begin{remark}
It is clear from inequality (\ref{mainIneq}) that Theorem \ref{th_main_nomin1} is a good generalization of classic Chern-Osserman inequality as long as the constant $C$ is not zero. This fact depends on the function $h(r)$, which bounds the radial mean curvature of the surface, as we try to explain with the following consideration. 

Let us consider $S$ a complete, connected and properly immersed surface in a Cartan-Hadamard manifold $N$ with pole $o \in N$ and with radial sectional curvatures $K_{o,N}(\sigma_x) \leq b <0$. Let us assume that hypotheses $I$, $II$ and $III$ in Theorem \ref{th_main_nomin1} are fulfilled.

Let us suppose that, for some fixed constant $L \geq 1$, and for all $x\in S$, the bound for the mean curvature of $S$ is given by

\begin{equation}
C(x)= -\langle\nabla^{N}r(x), H_{S}(x) \rangle  \leq h_L(r(x))=\frac{\sqrt{-b}}{L}e^{-2\sqrt{-b}r(x)}
\end{equation}

Then, it is straightforward that $C_{w_b,h_L}^{2}$ is strongly balanced and so, by applying Theorem \ref{th_main_nomin1}, $S^{2}$ has finite topological type, and
\begin{equation}
\begin{aligned}
-\chi(S)  &  \leq\frac{1}{2\pi}\int_{S}(b-K_{N})d\sigma+\frac{1}{4\pi}\int
_{S}\Vert A^{S}\Vert^{2}d\sigma\\-
&  C_{L}\sup_{t>0}\frac{\operatorname{Vol}(D_{t})}{\operatorname{Vol}
(B_{t}^{b,2})}+\frac{\sqrt{-b}}{\pi}\int_{S}\Vert H_S\Vert d\sigma-\frac
{1}{\pi}\int_{S}\Vert H_S\Vert^{2}d\sigma-\frac{bv(t_{0})}{2\pi}\nonumber
\end{aligned}
\end{equation}
where it is straightforward to check that 
$$1 \geq C_{L}=\underset{t>0}{\inf}\left(  \cosh\sqrt{-b}t-q_{W_L}(t)\sqrt{-b}
\sinh\sqrt{-b}t\right)\geq 1- \frac{2}{3L} \geq \frac{1}{3} >0$$

Hence,
$$\lim_{L \to \infty} C_L= 1$$

But, on the other hand, when $L$ goes to infinity, then $h_L(r)$ goes to the constant function $0$. 

In turn, it is straightforward to check that the value $t_0(b,h)$ (which ultimately depends on the model space $C^2_{w_b,h}$) satisfies in this case the following inequality:
\begin{equation}
t_{0}(b,h_L)\leq\frac{2}{\sqrt{-b}}arcsech\frac{\sqrt{L}}{\sqrt{L+1}}.
\end{equation}
and hence
\begin{equation*} 
\lim_{L\rightarrow+\infty}t_{0}(b,h_L)   \leq\lim_{L\rightarrow+\infty}\frac
{2}{\sqrt{-b}}arcsech\frac{\sqrt{L}}{\sqrt{L+1}}=0
\end{equation*}

Therefore, the minimal case can be considered not only a corollary but also a limit case, when $L \to \infty$, of the assumptions in Theorem \ref{th_main_nomin1} (given a suitable choice of the bounding function $h(r)=h_L(r)$).

\end{remark}



\section{Hessian analysis, Gauss-Bonnet Theorem, and estimates for the Euler characteristic of the extrinsic balls}


\subsection{Hessian and Laplacian comparison analysis}\label{subsecLap} 

We now assume that $S^{2}$ is a complete, non-compact, and properly immersed
surface in a Riemannian manifold $N^{n}$ which possesses a pole $o$.

The 2nd order analysis of the restricted distance function $r_{|_{S}}$
is governed by the Hessian comparison Theorem A in \cite{GW}. A corollary of this result is the
following proposition (see \cite{HP} or \cite{Pa3} for further details):

\begin{proposition}
\label{corLapComp} Let $N^{n}$ be a manifold with a pole $o$, and let $M_{\omega
}^{m}$ denote a $\omega-$model with center $o_{\omega}$. Let us suppose that
every $o$-radial sectional curvature at $x\in N-\{o\}$ is bounded from above
by the $o_{\omega}$-radial sectional curvatures in $M_{\omega}^{m}$ as
follows:
\begin{equation}
\mathcal{K}(\sigma(x))\,=\,K_{o,N}(\sigma_{x})\leq-\frac{\omega^{\prime\prime
}(r)}{\omega(r)} \label{eqKbound}
\end{equation}
for every radial two-plane $\sigma_{x}\subset T_{x}N$ at distance
$r=r(x)=\operatorname{dist}_{N}(o,x)$ from $p$ in $N$\medskip

For every smooth function $f(r)$ with $f^{\prime}(r)\,\geq\,(\leq)\,0\,\,\text{for
all}\,\,\,r$, and given $X\in TqS$ unitary:
\begin{equation}
\begin{aligned} \operatorname{Hess}^{S}(f\circ r)(X,X)&\geq(\leq)\,\left( \,f^{\prime \prime}(r)-f^{\prime}(r)\eta_{w}(r)\,\right) <X,\nabla^{N}r>^{2}\\ & +f^{\prime}(r)\left( \,\eta_{w}(r)+\langle\,\nabla^{N}r,\,A^{S} (X,X)\,\rangle\,\right) \end{aligned} \label{HessFunc2}
\end{equation}
\newline Tracing inequality (\ref{HessFunc2})
\begin{equation}
\begin{aligned} \Delta^{S}(f\circ r)\,& \geq(\leq)\,\left( \,f^{\prime\prime}(r)-f^{\prime }(r)\eta_{w}(r)\,\right) \Vert\nabla^{S}r\Vert^{2}\\ & +mf^{\prime}(r)\left( \,\eta_{w}(r)+\langle\,\nabla^{N}r,\,H_S \,\rangle\,\right) \end{aligned} \label{LapFunc2}
\end{equation}

\end{proposition}

\begin{proposition}
\label{lema 3.1 Chen}(See \cite{Che3} and \cite{EP}) Let $S^{2}$ be a
complete, non-compact, and properly immersed surface in a Cartan-Hadamard
manifold $N^{n}$. Let us consider $\{D_{t}\}_{t>0}$ an exhaustion of $S$ by
extrinsic balls. Let $f:S\rightarrow\mathbb{R}$ be a positive $C^{\infty}$
function. Then
\[
\int_{S}e^{-\sqrt{-b}r(x)}~f(x)d\sigma<+\infty\,\,\,\,\text{if and only
if}\,\,\int_{0}^{+\infty}e^{-\sqrt{-b}t}\int_{D_{t}}f(x)~d\sigma~dt<+\infty.
\]
\newline
\end{proposition}

\subsection{An application of the Gauss-Bonnet theorem: geodesic curvature of the extrinsic curves in the surface $S$}
\begin{proposition}
\label{GeodCurvPropNoMin} Let $S^{2}$ be a properly immersed surface in
a Cartan-Hadamard manifold $N$. Let us assume that, given a pole $o \in N$, the $o$-radial sectional curvatures of $N$ are bounded from above by a negative quantity $K_{o,N}(\sigma_x) \leq b <0$. Let $D_{t}$ be an extrinsic ball in $S$ centered on the pole $o \in N$. The
geodesic curvature of the extrinsic sphere $\partial D_{t}$, denoted as
$k_{g}^{t}$, is bounded from below as follows:
\begin{equation} \label{geodcurvineqNoMin}
\begin{aligned}
k_{g}^{t}  &  \geq\frac{1}{\Vert \nabla^{S}r\Vert }\left\{
\eta_{\omega_{b}}(t)+2\left\langle H_S,\nabla^{N}r\right\rangle
-\left\langle A^{S}(\frac{\nabla^{S}r}{\Vert \nabla^{S}r\Vert
},\frac{\nabla^{S}r}{\Vert \nabla^{S}r\Vert }),\nabla^{\perp
}r\right\rangle \right\}  \geq\\
&  \frac{1}{\Vert \nabla^{S}r\Vert }\left\{  \eta_{\omega_{b}
}(t)-2\Vert H_S\Vert -\left\langle A^{S}(\frac{\nabla^{S}
r}{\Vert \nabla^{S}r\Vert },\frac{\nabla^{S}r}{\Vert
\nabla^{S}r\Vert }),\nabla^{\perp}r\right\rangle \right\} 
\end{aligned}
 \end{equation}
where $A^S$ denotes the second fundamental form of $S$ in $N$.
\end{proposition}
\begin{proof}
As $K_N\leq b$ by applying (\ref{LapFunc2}) to the radial function $f(r)=\cosh\sqrt{-b}r$ and having into account that

\begin{equation}
-\Vert H_S\Vert \leq\left\langle \nabla^{N}r,H_S\right\rangle
\leq\Vert H_S\Vert
\end{equation}
then
\begin{equation}\label{compa_cosh}
\Delta^{S}\cosh\sqrt{-b}r\geq-2b\cosh\sqrt{-b}r-2\sqrt{-b}\sinh\sqrt
{-b}r\Vert H_S\Vert .
\end{equation}

Now, we again apply Proposition \ref{corLapComp} to $f(r)=1 \,\,\,\forall \,r \geq 0$ to conclude that the geodesic curvature $k_{g}^{t}$ satisfies the inequality 
\begin{equation}
\begin{aligned}
k_{g}^{t}  &  =\frac{1}{\Vert \nabla^{S}r\Vert }Hess^{S}
r(e,e)\geq\\
&  \frac{1}{\Vert \nabla^{S}r\Vert }\left\{  -\eta_{\omega_{b}
}\left\langle e,\nabla^{N}r\right\rangle ^{2}+\eta_{\omega_{b}}+\left\langle
A^{S}(e,e),\nabla^{N}r\right\rangle \right\}  =\\
&  \frac{1}{\Vert \nabla^{S}r\Vert }\left\{  \eta_{\omega_{b}
}+\left\langle A^{S}(e,e),\nabla^{N}r\right\rangle \right\}  ,
\end{aligned}
\end{equation}

\noindent where $e$ is unitary and tangent to $\partial D_{r}$.

As
\begin{equation}
H_S=\frac{1}{2}\left[  A^{S}(e,e)+A^{S}(\frac{\nabla^{S}r}{\Vert
\nabla^{S}r\Vert },\frac{\nabla^{S}r}{\Vert \nabla^{S}r\Vert
})\right]  ,
\end{equation}

we obtain:

\begin{equation}
\begin{aligned}
k_{g}^{t}  &  \geq\frac{1}{\Vert \nabla^{S}r\Vert }\left\{
\eta_{\omega_{b}}(t)+2\left\langle H_S,\nabla^{N}r\right\rangle
-\left\langle A^{S}(\frac{\nabla^{S}r}{\Vert \nabla^{S}r\Vert
},\frac{\nabla^{S}r}{\Vert \nabla^{S}r\Vert }),\nabla^{\perp
}r\right\rangle \right\}  \geq\\
&  \frac{1}{\Vert \nabla^{S}r\Vert }\left\{  \eta_{\omega_{b}
}(t)-2\Vert H_S\Vert -\left\langle A^{S}(\frac{\nabla^{S}
r}{\Vert \nabla^{S}r\Vert },\frac{\nabla^{S}r}{\Vert
\nabla^{S}r\Vert }),\nabla^{\perp}r\right\rangle \right\}  .
\end{aligned}
\end{equation}

\end{proof}
\begin{proposition}
\label{ChaEulerPropNoMin} Let $S^{2}$ be a properly immersed surface in
a Cartan-Hadamard manifold $N$. Let us assume that, given a pole $o \in N$, the $o$-radial sectional curvatures of $N$ are bounded from above by a negative quantity $K_{o,N}(\sigma_x) \leq b <0$. Let $D_{t}$ be an extrinsic ball in $S$ centered on the pole $o \in N$. The
volume $v(t)=\text{vol}(D_{t})$ satisfies the inequality:
\begin{equation}\label{ChaEulerPropIneqNoMin}
\begin{aligned}
&  2\pi\chi(D_{t})\geq\\
&  \int_{\partial D_{t}}\frac{1}{\Vert \nabla^{S}r\Vert }\left\{
\eta_{\omega_{b}}(t)-2\Vert H_S\Vert -\left\langle A^{S}
(\frac{\nabla^{S}r}{\Vert \nabla^{S}r\Vert },\frac{\nabla^{S}
r}{\Vert \nabla^{S}r\Vert }),\nabla^{\perp}r\right\rangle \right\}
d\sigma\\&+\int_{D_{t}}K_{S}d\sigma.
\end{aligned}
\end{equation}
\noindent where $K_{S}$ denotes the Gaussian curvature of $S$.
\end{proposition}
\begin{proof}
By applying the Gauss-Bonnet theorem:

\begin{equation}
\int_{\partial D_{t}}k_{g}^{t}d\mu+\int_{D_{t}}K_{S}d\sigma=2\pi\chi(D_{t}),
\end{equation}

Now, using Proposition \ref{GeodCurvPropNoMin}
\begin{equation}
\begin{aligned}
&  2\pi\chi(D_{t})\geq\\
&  \int_{\partial D_{t}}\frac{1}{\Vert \nabla^{S}r\Vert }\left\{
\eta_{\omega_{b}}(t)-2\Vert H_S\Vert -\left\langle A^{S}
(\frac{\nabla^{S}r}{\Vert \nabla^{S}r\Vert },\frac{\nabla^{S}
r}{\Vert \nabla^{S}r\Vert }),\nabla^{\perp}r\right\rangle \right\}
d\mu\\&+\int_{D_{t}}K_{S}d\sigma.
\end{aligned}
\end{equation}
\end{proof}

\begin{proposition}
\label{ChaEulerPropNoMin2} Let $S^{2}$ be a properly immersed surface in
a Cartan-Hadamard manifold $N$. Let us assume that, given a pole $o \in N$, the $o$-radial sectional curvatures of $N$ are bounded from above by a negative quantity $K_{o,N}(\sigma_x) \leq b <0$. Let $D_s \subset D_{t}$ be extrinsic balls in $S$ centered on the pole $o \in N$.
 Then
 \begin{equation}
 \begin{aligned}
&  \frac{\int_{D_{t}}\left(  \cosh\sqrt{-b}r-\Vert H_S\Vert
\frac{\sinh\sqrt{-b}r}{\sqrt{-b}}\right)  d\sigma}{\cosh^{2}\sqrt{-b}t}
\\&-\frac{\int_{D_{s}}\left(  \cosh\sqrt{-b}r-\Vert H_S\Vert
\frac{\sinh\sqrt{-b}r}{\sqrt{-b}}\right)  d\sigma}{\cosh^{2}\sqrt{-b}s}\geq\\
&  \int_{D_{t}-D_{s}}\frac{1+\sinh^{2}\sqrt{-b}r\Vert \nabla^{\perp
}r\Vert ^{2}-\frac{\sinh\sqrt{-b}r\cosh\sqrt{-b}r}{\sqrt{-b}}\Vert
H_S\Vert }{\cosh^{3}\sqrt{-b}r}d\sigma
\end{aligned}
\end{equation}
\end{proposition}
\begin{proof}
We integrate inequality (\ref{compa_cosh}), and then we apply the divergence theorem to obtain
\begin{equation}\label{designablar}
\begin{aligned}
  \sqrt{-b}\sinh\sqrt{-b}t&\int_{\partial D_{t}}\Vert \nabla
^{S}r\Vert d\sigma\geq\\
  -2b\int_{D_{t}}\cosh\sqrt{-b}s~d\sigma&-2\sqrt{-b}\int_{D_{t}}\Vert
H_S\Vert \sinh\sqrt{-b}s~d\sigma
\end{aligned}
\end{equation}
\newline Therefore
\begin{equation}
\begin{aligned}\label{desig2}
&  \int_{D_{t}}\left(  \cosh\sqrt{-b}r-\frac{\Vert H_S\Vert
\sinh\sqrt{-b}r}{\sqrt{-b}}\right)  ~d\sigma\leq\\
&  \frac{1}{2}\frac{\sinh\sqrt{-b}t}{\sqrt{-b}}\int_{\partial D_{t}}\Vert
\nabla^{S}r\Vert d\sigma_t
\end{aligned}
\end{equation}
\newline Deriving and using the above inequality 
\begin{equation*}
\begin{aligned}
&  \frac{d}{dt}\left(  \frac{\int_{D_{r}}\left(  \cosh\sqrt{-b}r-\Vert
H_S\Vert \frac{\sinh\sqrt{-b}r}{\sqrt{-b}}\right)  d\sigma}{\cosh
^{2}\sqrt{-b}t}\right)  \geq\\
&  \frac{1}{\cosh^{3}\sqrt{-b}t}\left\{  \int_{\partial D_{t}}\frac{\cosh
^{2}\sqrt{-b}r-\frac{\sinh\sqrt{-b}r\cosh\sqrt{-b}r}{\sqrt{-b}}\Vert
H_S\Vert -\sinh^{2}\sqrt{-b}r\Vert \nabla^{S}r\Vert ^{2}
}{\Vert \nabla^{S}r\Vert }d\mu\right\}  =\\
&  \int_{\partial D_{t}}\frac{1}{\Vert \nabla^{S}r\Vert }\left\{
\frac{1+\sinh^{2}\sqrt{-b}r\Vert \nabla^{\perp}r\Vert ^{2}
-\frac{\sinh\sqrt{-b}r\cosh\sqrt{-b}r}{\sqrt{-b}}\Vert H_S\Vert
}{\cosh^{3}\sqrt{-b}t}d\mu\right\}
\end{aligned}
\end{equation*}

Now, integrate the above inequality, applying the co-area formula.
\end{proof}

As direct corollaries of the above Propositions, we have the corresponding results for the minimal case (see \cite{Ch1} and \cite{Che3}), where $H_S=0$ and hence $\Vert H_S\Vert=0$.

\section{Extrinsic isoperimetry, volume growth and topology of surfaces}


\subsection{Extrinsic isoperimetry and the monotonicity property of the volume growth for non-minimal surfaces}


  In this section we are going to see how it is possible to deduce a monotonicity property satisfied by the volume growth function in the strongly balanced setting defined in section \S.3.

We start by studying how to obtain the classic monotonicity property of the volume growth function (see \cite{A2} and \cite {MP}) using a slightly more general isoperimetric inequality than the one used in \cite{MP}. This isoperimetric comparison is based, in turn, on a balance condition that is more general than the one used in \cite{MP}.


\begin{theorem}
\label{IsopTh} Let us consider $N^n$ to be a Cartan-Hadamard manifold, and $o \in N$ a pole in $N$. Let us suppose that its radial sectional curvatures are bounded from above by a negative bound $K_{o,N}(\sigma_x) \leq b <0$. Let $S$ be a complete, connected and properly immersed surface in $N$ such that there exists a radial function $h(r)$ satisfying:
 \begin{equation}
C(x)= -\langle\nabla^{N}r(x), H_{S}(x) \rangle \leq h(r(x)) \quad{\text{for all}
}\quad x \in S \,\, 
\end{equation}
where $H_{S}(x)$ denotes the mean curvature vector of $S$ in $N$.

Let $C_{\omega_{b},h}^{2}$ denote the $W$-model constructed via $\omega_{b}$ and $h$, and assume that $C_{\omega_{b},h}^{2}$ is a strongly balanced isoperimetric
comparison space on the interval $[\,0,\infty\,)$. Then, there exists $t_0 \geq 0$ such that
\begin{equation} \label{IsopIneq}
\frac{vol(\partial D_{R})}{vol(D_{R})-vol(D_{t_{0}})}\geq\frac{vol(\partial
B_{R}^{W})}{vol(B_{R}^{W})}\,\,\,,\text{~}\forall R\geq t_{0}. 
\end{equation}

\end{theorem}

\begin{proof}
We shall show the following two lemmas first:
\begin{lemma}\label{FirstProp}
If the isoperimetric comparison space $C_{\omega_{b},h}^{2}$ is strongly balanced, then

\begin{enumerate}
\item The function $h(t)$ satisfies
\begin{equation}
 \underset{t\rightarrow+\infty}{\lim}h(t)=0
 \end{equation}

\item The function $q_{W}(t)=\frac{\int_{0}^{t}W(s)ds}{W(t)}$ satisfies
\begin{equation}\label{ineq2}
q_W(t) \leq \frac{1}{\sqrt{-b}}\quad \forall t>0
\end{equation}

\begin{equation}\label{limits}
\begin{aligned}
\underset{t\rightarrow+\infty}{\lim}W(t)&=+\infty\\
 \underset{t\rightarrow+\infty}{\lim}q_{W}(t)&=\frac{1}{\sqrt{-b}}\\
 \underset{t\rightarrow0^{+}}{\lim}q_{W}(t)&=0
 \end{aligned}
\end{equation}

\end{enumerate}
\end{lemma}
\begin{proof}
As $0 \leq \lim _{t \to \infty} \vert h(t)\vert \leq\frac{1}{2} \lim_{t \to \infty} (\eta_{w_b}(t)-\sqrt{-b})=0$, we have that 
$$\lim _{t \to \infty} h(t)=0$$
To see (\ref{ineq2}), we use the fact that $ h(r) \leq \vert h(r)\vert \leq \frac{1}{2} (\eta_{w_b}(r)-\sqrt{-b})$ for all $r \geq 0$, and equation (\ref{eqLambdaDiffeq1}).

To show the limits in (\ref{limits}), we use the fact that $\lim _{t \to \infty} h(t)=0$. Therefore, it is straightforward to check that  $\lim _{t \to \infty}W(t)=+\infty$ and, hence, to apply L'Hospital's rule in order to obtain the other two limits.

\end{proof}


\begin{lemma}\label{Winfy} 
 Let us consider an isoperimetric comparison space $C_{\omega_{b},h}^{2}$. 
If $C_{\omega_{b},h}^{2}$ is strongly balanced on $[0, \infty)$, then there exists some $t_0 \geq 0$ such that the following inequality
holds for all $r\in\,[\,t_{0},\infty\,)$:
\begin{equation}\label{eqBalA}
q_{W}(r)\left(  \eta_{w_b}(r)-h(r)\right)\geq \frac{1}{2}
\end{equation}
where $q_{W}(r)$ is the isoperimetric quotient function introduced in equation (\ref{eqDefq}).
\end{lemma}
\begin{proof}

Applying (\ref{limits}) in Lemma \ref{FirstProp}, we have
  $$\lim_{t \to \infty} q_{W}(t)(\eta_{\omega_{b}}(t)-2h(t))=\lim_{t \to \infty} q_{W}(t)\lim_{t \to \infty} (\eta_{\omega_{b}}(t)-2h(t))=1$$

Hence, by applying the definition of limit when $t$ goes to infinity with $\epsilon=1/2$, we obtain that there exists $t_0\geq 0$ such that $q_{W}(t)(\eta_{\omega_{b}}(t)-2h(t)) \geq 1/2$.
\end{proof}

To show Theorem \ref{IsopTh}, let us now consider a fixed $R>t_{0}$. For all $t\in\lbrack t_{0},R]$, we define
\[
\psi(t)=\int_{t}^{R}\frac{1}{W(u)}\left(  \int_{0}^{u}W(s)ds\right)
du,~\forall t\geq t_{0}
\]

Using this definition and (\ref{eqDefq}) we have:
\begin{equation}
\begin{aligned} \psi^{\prime}(t) & =-q_{W}(t)=-\frac{vol(B_{t}^{W})}{vol(\partial B_{t} ^{W})}\leq0\\ \psi^{\prime\prime}(t) & =-1+q_{W}(t)(\eta_{\omega_{b}}(t)-2h(t)). \end{aligned} \label{deriv}
\end{equation}

We transplant $\psi$ to $S$, defining $\bar{\psi}:D_{R} - D_{t_{0}
}\rightarrow\mathbb{R}$ as $\bar{\psi}(x)=\psi(r(x))$

Applying (\ref{LapFunc2}) in Proposition \ref{corLapComp}:
\begin{equation}
\begin{aligned} \Delta^{S}\psi(r(x))&\leq(\psi^{\prime\prime}(r(x))-\psi^{\prime} (r(x))\eta_{\omega_{b}}(r(x)))\Vert \nabla^{S}r\Vert ^{2} \\&+2\psi^{\prime}(r(x))(\eta_{\omega_{b}}(r(x))-h(t))\quad. \end{aligned} \label{1}
\end{equation}

As $r(x)\geq t_{0}$, by applying the inequality (\ref{eqBalA}) in Lemma \ref{Winfy}, which holds for $\forall t\geq
t_{0}$, we obtain:
\[
\psi^{\prime\prime}(r(x))-\psi^{\prime}(r(x))\eta_{\omega_{b}}(r(x))\geq0,
\]
Hence as $\Vert\nabla^{S}r\Vert^{2}\leq1$ and using equations  (\ref{deriv}) and again inequality (\ref{eqBalA})  we have
$\Delta^{S}\psi(r(x))\leq-1.$

By integrating inequality (\ref{1}) on the annulus $A_{t_{0}}^{R}=D_{R}-D_{t_{0}}$ and applying the Divergence theorem, we obtain:
\begin{equation}
\begin{aligned} \operatorname{Vol}(A_{t_{0}}^{R})&\leq\int_{A_{t_{0}}^{R}}-\Delta^{S}\psi(r(x))d\mu \\&=-\psi^{\prime}(R)\int_{\partial D_{R}}\Vert \nabla^{S}r\Vert d\mu+\psi^{\prime}(t_{0})\int_{\partial D_{t_{0}}}\Vert \nabla ^{S}r\Vert d\mu\quad. \end{aligned}
\end{equation}
As $-\psi^{\prime}(t)=q_{W}(t)\geq0\,\,\,\forall t\geq0$, we have:
\begin{equation}
\begin{aligned} vol(D_{R})-vol(D_{t_{0}})&\leq q_{W}(R)vol(\partial D_{R}) \end{aligned}
\end{equation}
and hence $vol(D_{R})-vol(D_{t_{0}})\leq\frac{vol(B_{R}^{W})}{vol(S_{R}^{W}
)}vol(\partial D_{R})$  .
\end{proof}

As a first corollary, we obtain the comparison between the volume of extrinsic balls in the surface and the volume of the geodesic balls in the model space.

\begin{corollary}[General Monotonicity]\label{no decrecimiento}
Let us consider $N^n$ to be a Cartan-Hadamard manifold, and $o \in N$ a pole in $N$. Let us suppose that its radial sectional curvatures are bounded from above by a negative bound $K_{o,N}(\sigma_x) \leq b <0$. Let $S$ be a complete, connected and properly immersed surface in $N$ such that there exists a radial function $h(r)$ satisfying
 \begin{equation}
C(x)= -\langle\nabla^{N}r(x), H_{S}(x) \rangle \leq h(r(x)) \quad{\text{for all}
}\quad x \in S \,\, 
\end{equation}
where $H_{S}(x)$ denotes the mean curvature vector of $S$ in $N$.

Let $C_{\omega_{b},h}^{2}$ denote the $W$-model constructed via $\omega_{b}$ and $h$, and assume that $C_{\omega_{b},h}^{2}$ is a strongly balanced isoperimetric
comparison space on the interval $[\,0,\infty\,)$.
Then the function
\begin{equation*}
\frac{v(t)-v_{0}}{\Vol\left(B_{t}^{W}(o_{W}) \right )}
\end{equation*}
is non-decreasing in $[t_0, +\infty)$, where $t_0$ is given in Lemma \ref{Winfy}, being $v(t)=\Vol(D_t)$ and $v_0=v(t_0)=\Vol(D_{t_0})$.
\end{corollary}

\begin{proof}
Let us consider the functions $f(t)=\frac{vol(D_{t})-v_{0}}{vol(B_{t}^{W})}$ and $G(t)=\ln f(t)$

From the co-area formula:
\begin{equation}
v^{\prime}(t)=\int_{\partial D_{t}}\frac{1}{\Vert \nabla^{S}r\Vert
}d\mu\geq\int_{\partial D_{t}}d\mu=vol(\partial D_{t}),
\end{equation}
so
\begin{equation}
\frac{d}{dt}vol(D_{t})\geq vol(\partial D_{t}).
\end{equation}
On the other hand, in a rotationally symmetric space $M_W$ we have that, (see  \cite{Gri}):
\begin{equation}
vol(B_{t}^{W})^{\prime}=vol(\partial B_{t}^{W}).
\end{equation}

Hence, by applying Theorem \ref{IsopTh}, we obtain:
\begin{equation}
\begin{aligned}
G^{\prime}(t)&=\frac{v^{\prime}(t)}{v(t)-v_{0}}-\frac{vol(\partial B_{t}^{W}
)}{vol(B_{t}^{W})}\\ &\geq\frac{vol(\partial D_{t})}{v(t)-v_{0}}-\frac
{vol(\partial B_{t}^{W})}{vol(B_{t}^{W})} \geq 0\,\,\,\forall t\geq t_0
\end{aligned}
\end{equation}
so we have $f'(t) \geq 0\,\,\,\forall t \geq t_0$.
\end{proof}

Now, we are going to obtain two new
monotonicity properties deduced from the isoperimetric inequality
(\ref{IsopIneq}) in Theorem \ref{IsopTh}. The key difference with the generalized monotonicity property analyzed in Corollary \ref{no decrecimiento} is that now we want to compare the volume of the extrinsic $r$-balls with the hyperbolic cosine (as in the minimal context given in \cite{MP}), and not with the volume of the geodesic $r$-balls in the model space (as is performed in Corollary \ref{no decrecimiento}).

\begin{corollary}
[Non-minimal Monotonicity]\label{v entre cos}

Let us consider $N^n$ to be a Cartan-Hadamard manifold, and $o \in N$ a pole in $N$. Let us suppose that its radial sectional curvatures are bounded from above by a negative bound $K_{o,N}(\sigma_x) \leq b <0$. Let $S$ be a complete, connected and properly immersed surface in $N$ such that there exists a radial function $h(r)$ satisfying
 \begin{equation}
C(x)= -\langle\nabla^{N}r(x), H_{S}(x) \rangle \leq h(r(x)) \quad{\text{for all}
}\quad x \in S \,\, .
\end{equation}
where $H_{S}(x)$ denotes the mean curvature vector of $S$ in $N$.

Let $C_{\omega_{b},h}^{2}$ denote the $W$-model constructed via $\omega_{b}$ and $h$, and assume that $C_{\omega_{b},h}^{2}$ is a strongly balanced isoperimetric
comparison space on the interval $[\,0,\infty\,)$. Then, for some $t_0 \geq 0$, the function $\frac{v(t)-v_{0}}{\cosh\sqrt{-b}t-C}$ is
non-decreasing in $[t_{0},+\infty)$, where $t_0$ is given in Lemma \ref{Winfy} and the constant $C$ is defined as (see Theorem \ref{th_main_nomin1})
\begin{equation}
C=\underset{t>0}{\Inf}\left(  \cosh\sqrt{-b}t-q_{W}(t)\sqrt{-b}\sinh\sqrt
{-b}t\right)  .
\end{equation}

As a consequence, the function $\frac{v(t)-v_{0}}{\cosh\sqrt{-b}t}$ is
non-decreasing in $[t_{0},+\infty)$, where $v_0=\Vol(D_{t_0})$.
\end{corollary}

\begin{proof}

We are going to study the constant $C$ defined in the statement of the Theorem \ref{th_main_nomin1}.
To do so, we need the following consequence of Lemma \ref{FirstProp}:

\begin{lemma}\label{coroconstant} 
Let us consider an isoperimetric comparison space $C_{\omega_{b},h}^{2}$. Let us define the function $f(t):=\cosh\sqrt{-b}t-q_{W}(t)\sqrt{-b}\sinh
\sqrt{-b}t \quad \forall t>0$. Then $f(t)>0\,\,\forall t>0$ and $\lim_{t \to 0} f(t)=1$.
\end{lemma}
\begin{proof}
Applying (\ref{limits}) in Lemma \ref{FirstProp} again, we have: 
$$\lim_{t\rightarrow0^{+}}f(t)=1$$

 Finally, as $q_W(t) \leq \frac{1}{\sqrt{-b}}\,\,\forall t>0$, then:  
 $$f(t)\geq\cosh\sqrt{-b}
t-\sinh\sqrt{-b}t\geq0 \,\,\, \forall t>0$$
\end{proof}

Now, the proof of the theorem runs as follows: by applying Proposition \ref{Winfy} and Lemma \ref{coroconstant}, the function $f(t)$ in non-negative
and $\lim_{t\rightarrow0^{+}}f(t)=1$. Hence, the infimum $C$ exists, and we have
\begin{equation}\label{La_C}
0\leq C\leq 1
\end{equation}
Note that $C$ ultimately depends on the functions $h(r)$ and $\omega_b(r)$, namely $C=C_{h,b}$.

Now, we factor:
\begin{equation}
\frac{v(t)-v_{0}}{\cosh\sqrt{-b}t-C}=\frac{v(t)-v_{0}}{\int_{0}^{t}
W(s)ds}\frac{\int_{0}^{t}W(s)ds}{\cosh\sqrt{-b}t-C}
\end{equation}

The function $\frac{\int_{0}^{t}W(s)ds}
{\cosh\sqrt{-b}t-C}>0$ is non-decreasing for all $t \geq 0$ if and only if, for all $t \geq 0$

\begin{equation} 
W(t)(\cosh\sqrt{-b}t-C)-\sqrt{-b}\sinh\sqrt{-b}t \int_0^t W(s) ds \geq 0
\end{equation}

which is in turn equivalent to inequality $C \leq \cosh\sqrt{-b}t-q_W(t) \sqrt{-b}\sinh\sqrt{-b}t\,\,\,\forall t \geq 0$, which is true by definition of $C$.

On the other hand, and as $C_{\omega_{b},h}^{2}$ is strongly balanced, we apply Corollary \ref{no decrecimiento} to conclude that the function $\frac{v(t)-v_{0}}{\Vol\left(B_{t}^{W}(p_{W}) \right )}$ is non-decreasing in $[t_0, +\infty)$, for some $t_0 \geq 0$.

Therefore, we have the product of two positive and non-decreasing
functions in $[t_0, \infty)$, so the result is also non-decreasing in $[t_0, \infty)$, as
we wanted to prove.

Finally, the function $\frac{v(t)-v_{0}}{\cosh\sqrt{-b}t}$ is
non-decreasing in $[t_{0},+\infty)$. It follows directly from the fact that, for all $0 \leq C\leq 1$, and for all $t \geq t_0$,

\begin{equation}
\begin{aligned}
0 \leq & v'(t)(\cosh\sqrt{-b}t-C)-(v(t)-v_0)\sqrt{-b}\sinh\sqrt{-b}t \\& \leq v'(t)\cosh\sqrt{-b}t-(v(t)-v_0)\sqrt{-b}\sinh\sqrt{-b}t 
\end{aligned}
\end{equation}
\end{proof}





\begin{remark}
When the surface $S$ is minimal, it is used the function $h(r)=0$ as a radial controller for the mean curvature and the isoperimetric comparison space $C^2_{w_b,h}$ becomes the hyperbolic space $\Hatwo$. In this case $t_0=0$ and we have the isoperimetric inequality (see \cite{Pa} and \cite{MP5})
\begin{equation} \label{IsopIneq2}
\frac{vol(\partial D_{R})}{vol(D_{R})}\geq\frac{vol(\partial
B_{R}^{b,2})}{vol(B_{R}^{b,2})}\,\,\,,\text{~}\forall R\geq 0. 
\end{equation}

As a corollary of inequality (\ref{IsopIneq2}), we have the classic
monotonicity property for properly immersed minimal surfaces in Cartan-Hadamard manifolds with strictly negative curvature, (see \cite{A1} and \cite{MP}). In this case, the volume of the extrinsic balls is compared with the volume of the geodesic balls in the model space, $\Hatwo$, which is the hyperbolic cosine and we have that the function
$\frac{v(t)}{\Vol(B^{b,2}_t)}=\frac{v(t)}{\cosh(\sqrt{-b}t)-1}$ is non-decreasing in $[0, +\infty)$. This property also holds for minimal surfaces in the Euclidean spaces, (see \cite{MP}). 



\end{remark}


\subsection{Surfaces with finite topology}
On the other hand, we have the following theorem, which provides an extrinsic version of the proof of Huber's classical theorem given by White in \cite{W}. As we have mentioned in the Introduction, this is a key result that will allow us to argue as in \cite{Ch2} and \cite{Che3} (where it is possible to conclude that $\chi(S)=\lim_{t \to \infty} \chi(D_t)$ for an exhaustion of $S$ by extrinsic balls $\{D_t\}_{t>0}$).

Recall that an exhaustion of the surface $S$ is a sequence of subsets $\{D_t\subseteq S\}_{t>0}$ such that:
\begin{itemize}
\item $D_t \subseteq D_s$ when $s\geq t$
\item $\cup_{t>0} D_t=S$
\end{itemize}

\begin{theorem}\label{Huber}

Let $S^2$ be a complete, connected and oriented surface properly immersed in a Cartan-Hadamard manifold $N$. Let $\{D_{r_i}\}_{i=i}^\infty$ be an exhaustion of $S$ by extrinsic balls centered at a pole $o \in N$, where $\{r_i\}_{i=1}^{\infty}$ is an increasing sequence of extrinsic radius such that $r_i \to \infty$ when $i \to \infty$. If we have:
\begin{equation*}
\lim_{i\rightarrow \infty}\inf(\{-\chi(D_{r_k})\}_{k=i}^\infty)<\infty
\end{equation*}

Then, 

\begin{enumerate}

\item $S^2$ has finite topology, and 
\item $-\chi(S)\leq\lim_{i\rightarrow \infty} \inf(\{-\chi(D_{r_k})\}_{k=i}^\infty)$
\end{enumerate}
\end{theorem}

\begin{proof}
As the extrinsic balls $D_{r}$ in a properly immersed and connected submanifold $S$ are precompact and connected sets, we have
$$
-\chi(D_{r})=2g(r)+c(r)-2
$$
where $g(r)$ and $c(r)$ are the genus (number of handles), and the number of boundary components of $D_{r}$, respectively.

 Hence, if we consider $\{D_{r_i}\}_{i=i}^\infty$ to be the exhaustion of $S$ by extrinsic balls (where $\{r_i\}_{i=1}^{\infty}$ is an increasing sequence of extrinsic radius such that $r_i \to \infty$ when $i \to \infty$) which satisfies 
$\lim_{i\rightarrow \infty}\inf(\{-\chi(D_{r_k})\}_{k=i}^\infty)<\infty$, we have, taking limits: 
 
\begin{equation}
\begin{aligned}
&\lim_{i\rightarrow \infty}\inf(\{-\chi(D_{r_k}\}_{k=i}^\infty)\\&=2\lim_{i\rightarrow \infty}\inf(\{g(r_k)\}_{k=i}^\infty)+\lim_{i\rightarrow \infty}\inf(\{c(r_k)\}_{k=i}^\infty)-2<\infty
\end{aligned}
\end{equation}
Therefore, as $\lim_{i\rightarrow \infty}\inf(\{g(r_k)\}_{k=i}^\infty) \geq 0$, $\lim_{i\rightarrow \infty}\inf(\{c(r_k)\}_{k=i}^\infty) \geq 0$ and $g(r)$ is a non-decreasing, integer-valued function of $r$,

\begin{equation}
\lim_{i\rightarrow \infty}\inf(\{g(r_k)\}_{k=i}^\infty)=\lim_{i\rightarrow \infty}g(r_i)=g<\infty
\end{equation}

As $c(r)$ is also an integer-valued function of $r$, 
\begin{equation}
\lim_{i\rightarrow \infty}\inf(\{c(r_k)\}_{k=i}^\infty)<\infty
\end{equation}

 On the other hand, as $\lim_{i\rightarrow \infty}\inf(\{c(r_k)\}_{k=i}^\infty)<\infty$ and $\{c(r_k)\}_{k=i}^\infty$ is a sequence of natural numbers, then for each $i \in \N$ there exists a natural number $l(i)$, $l(i)\geq i$, such that
$$
\inf(\{c_k\}_{k=i}^\infty)=c_{l(i)}
$$
and hence
$$
\lim_{i\rightarrow \infty}c_{l(i)}=c<\infty
$$

Summarizing, there exists a natural number $k_0$ such that, for all $i \geq k_0$
$$
\begin{aligned}
g(r_i)&=g\\
c(r_i)&=c
\end{aligned}
$$
and, therefore, such that any compact subset of $S\setminus D_{r_{k_0}}$ has a genus equal to zero.

Now, given the sequence $\{r_i\}_{i=k_0}^{\infty}$ and for each $r_i$, let $A_i$ be the union of $D_{r_i}$ with those connected components of $S \setminus D_{r_i}$ which are compact (if there are none, then $A_i=D_{r_i}$). Let $g(A_i)$ and $c(A_i)$ denote the number of handles and boundary components, respectively, of $A_i$. As $A_i$ is precompact itself, then, provided $j \geq i$ is large enough
$$
D_{r_i} \subseteq A_i \subseteq D_{r_j}
$$

Hence, as $j >i\geq k_0$, $g=g(r_i)\leq g(A_i) \leq g(r_j)=g$, so: 
\begin{equation}\label{genus}
g(A_i)=g\,\,\,\forall i\geq k_0
\end{equation}
\noindent and, by construction of $A_i$, we also have that $c(A_i) \leq c(r_i)\,\,\,\forall i\geq k_0$, so additionally we can conclude that 
\begin{equation}\label{boundary}
c(A_i) \leq c\,\,\forall i \geq k_0
\end{equation}

As a consequence of (\ref{genus}) and (\ref{boundary}), we have that the $A_i$, ($i \geq k_0$), are homeomorphic, with $A_{i+1}$ obtained from $A_i$ by attaching annuli. 

Therefore, $S$ has finite topology, because $S=A_{k_0} \cup S \setminus A_{k_0}$, and $A_{k_0}$ is compact and $S \setminus A_{k_0}$ is homeomorphic to a finite union of cylinders. Moreover:
\begin{equation}
\chi(S)=\chi((S\setminus A_{k_0}) \cup A_{k_0})=\chi(S\setminus A_{k_0})+\chi(A_{k_0}) =\chi(A_{k_0}) 
\end{equation}

so, as $g(A_k)=g(r_k)=g$ and $c(A_k)\leq c(r_k)\leq c$, 
$$\chi(S)=\chi(A_k) \geq 2-2g-c$$

and therefore:
\begin{equation}
\begin{aligned}
&\lim_{i\rightarrow \infty}\inf(\{-\chi(D_{r_k}\}_{k=i}^\infty)\\&=2\lim_{i\rightarrow \infty}\inf(\{g(r_k)\}_{k=i}^\infty)+\lim_{i\rightarrow \infty}\inf(\{c(r_k)\}_{k=i}^\infty)-2\\&= 2g+c-2\geq -\chi(S)
\end{aligned}
\end{equation}
\end{proof}


\section{Proof of main Theorem}
\medskip

Let us consider $\{D_{t}\}_{t>0}$ to be an exhaustion of $S$
by extrinsic balls, centered at a pole $o \in N$.

Let us denote
\begin{equation}
I(t)=\int_{\partial D_{t}}\left\langle A^{S}(\frac{\nabla^{S}r}{\Vert
\nabla^{S}r\Vert},\frac{\nabla^{S}r}{\Vert\nabla^{S}r\Vert}),\frac
{\nabla^{\perp}r}{\Vert\nabla^{S}r\Vert}\right\rangle d\mu,
\end{equation}

Then we have, by applying Proposition \ref{ChaEulerPropNoMin}, the co-area formula, and adding and subtracting $b$\textperiodcentered$v(t)$ in inequality
(\ref{ChaEulerPropIneqNoMin}), 
\begin{equation}
2\pi\chi(D_{t})\geq\int_{D_{t}}(K_{S}-b)d\sigma+\eta_{\omega_{b}}v^{\prime
}(t)+b\text{\textperiodcentered}v(t)-2\int_{\partial D_{t}}\frac{\Vert
H_S\Vert}{\Vert\nabla^{S}r\Vert}d\sigma_{t}-I(t) \label{desig_inicial}
\end{equation}

As, for any $b<0$ and for all $t>0$,
\[
\eta_{\omega_{b}}(t)v^{\prime}(t)+b~v(t)=\sqrt{-b}\frac{\cosh^{2}(\sqrt{-b}
t)}{\sinh(\sqrt{-b}t)}\text{~}\frac{d}{dt}\frac{v(t)}{\cosh(\sqrt{-b}t)}
\]

then
\begin{equation}
\begin{aligned} \frac{d}{dt}\left( \frac{v(t)}{\cosh\sqrt{-b}t}\right) & \leq\frac{\sinh \sqrt{-b}t}{\sqrt{-b}\cosh^{2}\sqrt{-b}t}\left\{ 2\pi\chi(D_{t})+\int_{D_{t} }(b-K_{S})d\sigma+\right. \\ & \left. 2\int_{\partial D_{t}}\frac{\Vert H_S\Vert }{\Vert \nabla^{S}r\Vert }d\mu+I(t)\right\} \end{aligned} \label{desig_inicial1}
\end{equation}

Now, using that, for all $t \geq 0$, 
\begin{equation}\label{sin_exp}
\frac{\sinh(\sqrt{-b}t)}{\cosh^{2}(\sqrt{-b}t)}\leq2e^{-\sqrt{-b}t}
\end{equation}
so we obtain

\begin{equation}\label{desig_inicial2}
\begin{aligned} 
\frac{d}{dt}\left( \frac{v(t)}{\cosh\sqrt{-b}t}\right)&\leq\frac{1}{\sqrt{-b}}\left\{2e^{-\sqrt{-b}t}\int_{D_{t} }(b-K_{S})d\sigma+\frac{\sinh \sqrt{-b}t}{\cosh^{2}\sqrt{-b}t}I(t)+\right. \\ & \left. 4e^{-\sqrt{-b}t} \int_{\partial D_{t}}\frac{\Vert H_S\Vert }{\Vert \nabla^{S}r\Vert }d\mu+4e^{-\sqrt{-b}t}\pi\chi(D_{t})\right\} \end{aligned}
\end{equation}

As we have observed before, the extrinsic balls $D_{t}$ in a properly immersed and connected surface $S$ are connected, precompact domains. Hence, for all $t \geq 0$, we have:

\begin{equation}\label{des5}
\chi(D_{t})=2-2g(t)-c(t)\leq1
\end{equation}

\noindent where $g(t)$ and $c(t)$ are the genus (number of handles), and the number of boundary components of $D_{r}$, respectively (see
\cite{Mass}, p. 43).

Then, we integrate both sides of inequality (\ref{desig_inicial2}) between $0$ and a fixed $t >0$, having into account that $\frac{v(0)}{\cosh(0)}=0$ and applying the co-area formula:

\begin{equation}
\label{cuentagorda2}
\begin{aligned} &\frac{v(t)}{\cosh(\sqrt{-b}t)}\leq\frac{1}{\sqrt{-b}}\{2\int_{0}^{t}e^{-\sqrt{-b}s}\int_{D_{s}}(b-K_{S})d\sigma ds\\&+\int_{0}^{t}\frac{\sinh (\sqrt{-b}s)}{\cosh^{2}(\sqrt{-b}s)}I(s)ds+4\int_{0}^{t}\int_{\partial D_{s}}e^{-\sqrt{-b}s}\frac{\Vert H_S\Vert }{\Vert\nabla^{S}r\Vert}d\sigma_{s}ds\\&+4\pi\int_{0}^{t}\chi(D_{s})e^{-\sqrt{-b}s}ds \leq\frac{1}{\sqrt{-b}}\{2\int_{0}^{t}e^{-\sqrt{-b}s}\int_{D_{s}}(b-K_{S})d\sigma ds\\ &+\int_{0}^{t}\frac{\sinh(\sqrt{-b}s)}{\cosh^{2}(\sqrt{-b}s)}I(s)ds+4\int _{0}^{t}\int_{\partial D_{s}}e^{-\sqrt{-b}s}\frac{\Vert H_S\Vert}{\Vert\nabla^{S}r\Vert}d\sigma_{s}ds+C(0)\}\end{aligned}
\end{equation}
\noindent where
\begin{equation}
 C(0)= 4\pi \int_{0}^{\infty}e^{-\sqrt{-b} s} ds=\frac{4\pi}{\sqrt{-b}} <\infty 
\end{equation}
\noindent because, as $D_{s}$ is (pre)compact for all finite radii $s$, then
$\chi(D_{s})<\infty\quad\forall s$.

We are going to estimate $\operatorname{Sup}_{t>0} \frac{v(t)}{\cosh(\sqrt
{-b}t)}$ using the above inequality. To do so, we proceed as follows.

As $\int_{S}\Vert A^{S}\Vert^{2}d\sigma<+\infty$, then $\int_{S}e^{-\sqrt
{-b}r}\Vert A^{S}\Vert^{2}d\sigma<+\infty$, and similarly, using hypotheses (I) and (III), we have that $\int_{S}e^{-\sqrt
{-b}r}(b-K_N)d\sigma<+\infty$ and $\int_{S}e^{-\sqrt
{-b}r}\Vert H_S\Vert d\sigma<+\infty$.

Now, by applying Proposition \ref{lema 3.1 Chen} to the non-negative functions
$\Vert A^{S}\Vert^{2}$, $b-K_N(x)$, and $\Vert H_S\Vert$ we have:
\begin{equation}\label{acot1}
\begin{aligned}
\int_{0}^{+\infty}e^{-\sqrt{-b}t}R(t)~dt&<+\infty;\int_{0}^{+\infty}
e^{-\sqrt{-b}t}\int_{D_{t}}(b-K_N)d\sigma dt<+\infty\\& \int_{S}e^{-\sqrt{-b}r(x)}\Vert H_S\Vert d\sigma<+\infty
\end{aligned}
\end{equation}

With these estimates we can conclude, using equality (\ref{TheEquation}) in Remark \ref{remark3.4} and
the co-area formula, and taking into account that, for all $t \geq 0$, the quantity $-\frac{4}{\sqrt{-b}}\int_{0}^{t}e^{-\sqrt{-b}s}\int_{D_{s}}\Vert H_S\Vert^{2}d\sigma ds$ is negative:
\begin{equation}\label{cuentagorda3}
\begin{aligned} & \frac{v(t)}{\cosh(\sqrt{-b}t)}\leq\frac{2}{\sqrt{-b}}\int_{0}^{t}e^{-\sqrt{-b}s}\int_{D_{s}}(b-K_N)d\sigma ds\\&+\frac{1}{\sqrt{-b}}\int_{0}^{t}e^{-\sqrt{-b}s}\int_{D_{s}}\Vert A^{S}\Vert^{2}d\sigma ds\\-& \frac{4}{\sqrt{-b}}\int_{0}^{t}e^{-\sqrt{-b}s}\int_{D_{s}}\Vert H_S\Vert^{2}d\sigma ds+\frac{1}{\sqrt{-b}}\int_{0}^{t}\frac{\sinh(\sqrt{-b}s)}{\cosh^{2}(\sqrt{-b}s)}I(s)ds\\ & +\frac{4}{\sqrt{-b}}\int_{D_{s}}e^{-\sqrt{-b}r}\Vert H_S\Vert d\sigma +\frac{C(0)}{\sqrt{-b}}\\&\leq C_{1}(0)+\frac{1}{\sqrt{-b}}\int _{0}^{t}\frac{\sinh(\sqrt{-b}s)}{\cosh^{2}(\sqrt{-b}s)}I(s)ds \end{aligned} 
\end{equation}

\noindent where 
\begin{equation*}
\begin{split}
C_{1}(0)=\frac{1}{\sqrt{-b}}C(0)+\frac{2}{\sqrt{-b}
}\int_{0}^{+\infty}e^{-\sqrt{-b}t}\int_{D_{t}}(b-K_N)d\sigma dt\\+\frac
{1}{\sqrt{-b}}\int_{0}^{+\infty}e^{-\sqrt{-b}t}R(t)~dt+\frac{4}{\sqrt{-b}}\int_{S}e^{-\sqrt{-b}r}\Vert H_S\Vert d\sigma
\end{split}
\end{equation*} is a positive and
finite constant.

We now have the following result.

\begin{lemma}
\label{lemon} There exist two non-negative constants $C_{2}$ and $C_{3}$ such
that
\begin{equation}
\int_{0}^{t}\frac{\sinh\sqrt{-b}s}{\cosh^{2}\sqrt{-b}s}I(s)ds\leq C_{2}
\sqrt{C_{3}+\frac{v(t)}{\cosh\sqrt{-b}t}}
\end{equation}

\end{lemma}

\begin{proof}
Let us consider $\{e_1,e_2\}$ to be an orthonormal basis of $T_pS$, ($p \in S$), being $e_1=\frac{\nabla^S r}{\Vert \nabla^S r\Vert}$. Then
\begin{equation}
\Vert A^S(\frac{\nabla^S r}{\Vert \nabla^S r\Vert},\frac{\nabla^S r}{\Vert \nabla^S r\Vert})\Vert^2
\leq \Vert A^S\Vert^2 
\end{equation}
so 
\begin{equation}\label{ineqbot}
\langle A^{S}(\frac{\nabla^{S}r}{\Vert \nabla^{S}r\Vert
},\frac{\nabla^{S}r}{\Vert \nabla^{S}r\Vert }),\nabla^{\bot
}r\rangle \leq \Vert A^S\Vert ~\Vert\nabla^{\bot} r\Vert
\end{equation}

Applying Cauchy-Schwartz Inequality to the functions $\frac{\Vert A^S\Vert}{(\cosh (\sqrt{-b} r))^{1/2}}$ and $\frac{\sinh (\sqrt{-b} r) \Vert \nabla^{\bot} r\Vert}{(\cosh (\sqrt{-b} r))^{3/2}}$, we obtain:
\begin{equation}
\begin{aligned}
&  \int_{D_{t}}\frac{\sinh(\sqrt{-b}r)}{\cosh^{2}(\sqrt{-b}
r)}\left\langle A^{S}(\frac{\nabla^{S}r}{\Vert \nabla^{S}r\Vert
},\frac{\nabla^{S}r}{\Vert \nabla^{S}r\Vert }),\nabla^{\bot
}r\right\rangle d\sigma\leq\\
&  \int_{D_{t}}\sinh(\sqrt{-b}r)\Vert A^{S}\Vert \frac{\Vert
\nabla^{\bot}r\Vert }{\cosh^{2}(\sqrt{-b}r)}d\sigma 
\end{aligned}
\end{equation}

On the other hand, if we consider $s_0=0$ and $t_0=t$ in Proposition \ref{ChaEulerPropNoMin2}, as
$\cosh\sqrt{-b}r$ is non-decreasing, we have the following inequalities:

\begin{equation}
\begin{aligned} & \frac{\int_{D_{t}}\left( \cosh\sqrt{-b}r-\Vert H_S\Vert \frac{\sinh\sqrt{-b}r}{\sqrt{-b}}\right) d\sigma}{\cosh^{2}\sqrt{-b}t} \geq\\ & \\ &\int_{D_{t}}\frac{\sinh^{2}\sqrt{-b}r\Vert \nabla^{\perp }r\Vert ^{2}}{\cosh^{3}\sqrt{-b}r}d\sigma-\frac{1}{\sqrt{-b}} \int_{D_{t}}\frac{\sinh\sqrt{-b}r\Vert H_S\Vert }{\cosh^{2}\sqrt{-b}r}d\sigma \end{aligned}
\end{equation}
But as $\cosh\sqrt{-b}r$ is non-decreasing and $\frac{\int_{D_{t}}\sinh
\sqrt{-b}r\Vert H_S\Vert d\sigma}{\cosh^{2}\sqrt{-b}t}\geq0$, we have

\begin{equation}
\begin{aligned}
 &\frac{\int_{D_{t}}\left(  \cosh\sqrt{-b}r-\Vert H_S\Vert
\frac{\sinh\sqrt{-b}r}{\sqrt{-b}}\right)  d\sigma}{\cosh^{2}\sqrt{-b}t} \leq \\& \frac{\int_{D_{t}}\left(  \cosh\sqrt{-b}r\right)  d\sigma}{\cosh^{2}\sqrt{-b}t} \leq  \frac{v(t)}{\cosh\sqrt{-b}t}
\end{aligned}
\end{equation}
and therefore
\begin{equation}
\begin{aligned} \int_{D_{t}}\frac{\sinh^{2}\sqrt{-b}r~\Vert \nabla^{\perp }r\Vert ^{2}}{\cosh^{3}\sqrt{-b}r}\leq\frac{v(t)}{\cosh\sqrt{-b}t}+\frac{1}{\sqrt{-b}}\int_{D_{t}}\frac{\sinh\sqrt {-b}r\Vert H_S\Vert }{\cosh^{2}\sqrt{-b}r}d\sigma \end{aligned}
\end{equation}

Returning to the main computation in the Lemma, taking into account
that $\frac{1}{\cosh\sqrt{-b}r}\leq2e^{-\sqrt{-b}r}$ and $\frac{\sinh\sqrt
{-b}r}{\cosh^{2}\sqrt{-b}r}\leq2e^{-\sqrt{-b}r}$, we have:

\begin{equation} \label{desi2}
\begin{aligned} & \int_{0}^{t}\frac{\sinh\sqrt{-b}s}{\cosh^{2}\sqrt{-b}s}I(s)ds\\&\leq \sqrt{\int_{D_{t}}\frac{\Vert A^{S}\Vert ^{2}}{\cosh\sqrt{-b}r}d\sigma }\sqrt{\int_{D_{t}}\frac{\sinh^{2}\sqrt{-b}r~\Vert \nabla^{\perp }r\Vert ^{2}}{\cosh^{3}\sqrt{-b}r}}d\sigma\leq\\ & \sqrt{\int_{D_{t}}2e^{-\sqrt{-b}r}\Vert A^{S}\Vert ^{2}d\sigma} \sqrt{\frac{v(t)}{\cosh\sqrt{-b}t}+\frac{1}{\sqrt{-b}}\int_{D_{t}}e^{-\sqrt{-b} r}\Vert H_S\Vert d\sigma} \end{aligned} 
\end{equation}
Applying hypotheses (II) and (III):
\begin{equation}
\int_{0}^{t}\frac{\sinh\sqrt{-b}s}{\cosh^{2}\sqrt{-b}s}I(s)ds\leq C_{2}
\sqrt{\frac{v(t)}{\cosh\sqrt{-b}t}+C_{3}}
\end{equation}
where

\begin{equation*}
\begin{aligned}
0\leq C_2&=\sqrt{\int_{D_{t}}2e^{-\sqrt{-b}r}\Vert A^{S}\Vert ^{2}d\sigma}<\infty\\ 
0\leq C_3&=\frac{1}{\sqrt{-b}}\int_{D_{t}}e^{-\sqrt{-b}
r}\Vert H_S\Vert d\sigma <\infty
\end{aligned}
\end{equation*}
and the Lemma is proven.
\end{proof}

By applying Lemma \ref{lemon} to inequality (\ref{cuentagorda3}), we obtain
\begin{equation}
\frac{v(t)}{\cosh\sqrt{-b}t}\leq C_{1}(0)+C_{2}\sqrt{\frac{v(t)}
{\cosh\sqrt{-b} t}+C_{3}}
\end{equation}

By putting $f(t)=\sqrt{\frac{v(t)}{\cosh\sqrt{-b}t}+C_{3}}$, the above inequality 
becomes 
$$f(t)^{2}-C_{2}f(t)-(C_{3}+C_{1}(0))\leq0\,\forall t~$$
and hence the
values of $f(t)$ lie between the zeros of the function $g(x)=x^{2}%
-C_{2}x-(C_{3}+C_{1}(0))$, which are real and distinct numbers (because
$C_{1}(0)>0$, $C_{2},C_{3}\geq0$). Hence, $f(t)$ (and also $f^{2}(t)=\frac{v(t)}{\cosh\sqrt{-b}t}+C_{3}$)
are bounded, so therefore: 
$$\operatorname{Sup}_{t \geq 0}\frac{v(t)}{\cosh\sqrt{-b}t}=M<+\infty$$ 

On the other hand, from Corollary \ref{v entre cos} we know that $\frac{v(t)-v_0}{\cosh\sqrt{-b}t}$ is a non-decreasing function, so because for all $t \geq t_0$, $\frac{v(t)-v_0}{\cosh\sqrt{-b}t} \leq \frac{v(t)}{\cosh\sqrt{-b}t} \leq M$, we have that the limit $\lim_{t \to \infty}  \frac{v(t)}{\cosh\sqrt{-b}t}$ exists and

\begin{equation}\label{lim1}
\lim_{t \to \infty}  \frac{v(t)}{\cosh\sqrt{-b}t}=\lim_{t \to \infty} \frac{v(t)-v_0}{\cosh\sqrt{-b}t} \leq M <\infty
\end{equation}
and hence
\begin{equation}\label{lim2}
\lim_{t \to \infty} \frac{v(t)}{\cosh\sqrt{-b}t-1} =\lim_{t \to \infty}  \frac{v(t)}{\cosh\sqrt{-b}t} \leq M <\infty
\end{equation}

and

\begin{equation}\label{lim3}
\lim_{t \to \infty} \frac{v(t)}{\sinh\sqrt{-b}t} =\lim_{t \to \infty}  \frac{v(t)}{\cosh\sqrt{-b}t-1} \leq M <\infty
\end{equation}
\bigskip

To prove assertion (2) of the Theorem, we need the following:

\begin{lemma}\label{acot integral cos*v' no minimal} Let us consider $C$ to be the constant defined in Theorem \ref{th_main_nomin1}. Then
\begin{equation}\label{acot integral cos*v' no minimal1}
\int_{0}^{t}\cosh\sqrt{-b}t~v^{\prime
}(s)ds\geq\frac{\cosh\sqrt{-b}t+C}{2}v(t)-\frac{v_{0}(\cosh\sqrt{-b}t-1)}{2}
\end{equation}
\end{lemma}

\begin{proof}
By Corollary \ref{v entre cos} we have that $\frac{v(t)-v_{0}}{\cosh\sqrt{-b}t-C}$ is non-decreasing in $]t_{0}
,+\infty\lbrack$ and $v(t)$ is non-decreasing, so we
have:
\begin{equation}
(\cosh\sqrt{-b}t-C)v^{\prime}(t)\geq\left(  v(t)-v_{0}\right)  \sqrt{-b}
\sinh\sqrt{-b}t,\forall t\geq0.
\end{equation}
Integrating by parts:
\begin{align}
&  \int_{0}^{t}\cosh\sqrt{-b}t~v^{\prime}(s)ds\geq\\
&  v(t)\cosh\sqrt{-b}t-v_{0}\left(  \cosh\sqrt{-b}t-1\right)  +Cv(t)-\int
_{0}^{t}\cosh\sqrt{-b}t~v^{\prime}(s)ds.\nonumber
\end{align}
and we obtain the result by isolating $\int_{0}^{t}\cosh\sqrt{-b}t~v^{\prime
}(s)ds$.
\end{proof}

Once we have proven Lemma \ref{acot integral cos*v' no minimal}, we proceed as follows:

By definition of $I(t)$, inequality (\ref{ineqbot}), and the arithmetic-geometric mean inequality $xy\leq\frac{x^{2}+y^{2}}{2}$, we have: 
\begin{equation}\label{ineqdeI}
\begin{aligned}
  I(t)&\leq\int_{\partial D_{t}}\Vert A^{S}\Vert \frac{\Vert
\nabla^{\perp}r\Vert }{\Vert\nabla^{S}r\Vert}d\sigma_t\\&=\int_{\partial D_{t}
}\frac{\Vert A^{S}\Vert }{\sqrt{\eta_{\omega_{b}}(t)}\sqrt
{\Vert\nabla^{S}r\Vert}}\frac{\sqrt{\eta_{\omega_{b}}(t)}\Vert \nabla^{\perp}r\Vert }
{\sqrt{\eta_{\omega_{b}}(t)}\sqrt{\Vert\nabla^{S}r\Vert}}d\sigma_t\leq\\
&  \frac{1}{2}\int_{\partial D_{t}}\left(  \frac{\Vert A^{S}\Vert
^{2}}{\eta_{\omega_{b}(t)}\Vert\nabla^{S}r\Vert}+\frac{\eta_{\omega_{b}}(t)\Vert
\nabla^{\perp}r\Vert ^{2}}{\Vert\nabla^{S}r\Vert}\right) d\sigma_t\leq\\&
\frac{1}{\eta_{\omega_{b}}(t)}\int_{\partial D_{t}}\frac{\Vert
A^{S}\Vert ^{2}}{\Vert\nabla^{S}r\Vert}+\eta_{\omega_{b}}(t)\int
_{\partial D_{t}}\frac{\Vert \nabla^{\perp}r\Vert ^{2}}{\Vert
\nabla^{S}r\Vert}d\sigma_t. 
\end{aligned}
\end{equation}
But, on applying the co-area formula:
\[
\frac{1}{\eta_{\omega_{b}}(t)}R^{\prime}(t)=\frac{1}{\eta_{\omega_{b}}(t)}
\int_{\partial D_{t}}\frac{\Vert A^{S}\Vert ^{2}}{\Vert\nabla
^{S}r\Vert}d\sigma_t,
\]
so we have:
\begin{equation}\label{la_I}
I(t) \leq \frac{R'(t)}{\eta_{\omega_{b}}(t)}+ \eta_{\omega_{b}}(t)\int
_{\partial D_{t}}\frac{\Vert \nabla^{\perp}r\Vert ^{2}}{\Vert
\nabla^{S}r\Vert}d\sigma_t.
\end{equation}

We are now going to analyze the integral $\int_{\partial D_{t}}\frac{\Vert\nabla^{\perp}r\Vert^{2}}{\Vert\nabla
^{S}r\Vert}d\sigma_t$.

By integrating inequality (\ref{compa_cosh}) and applying the divergence theorem ($\sinh\sqrt{-b}s$ is increasing)
and the co-area formula:
\begin{equation} \label{W'_v'}
\int_{\partial D_{t}}\left\Vert \nabla^{S}r\right\Vert d\mu\geq\frac
{2\sqrt{-b}}{\sinh\sqrt{-b}t}\int_{0}^{t}\cosh\sqrt{-b}t~v^{\prime}
(s)d\sigma-2\int_{D_{t}}\left\Vert H_S\right\Vert d\sigma. 
\end{equation}
If we apply this inequality and the co-area formula:
\begin{equation}
\begin{aligned}
&  \eta_{\omega_{b}}(t)\int_{\partial D_{t}}\frac{\Vert\nabla^{\perp}
r\Vert^{2}}{\Vert\nabla^{S}r\Vert}d\sigma_{t}\leq\\
&  \eta_{\omega_{b}}(t)v^{\prime}(t)-\frac{2\sqrt{-b}\eta_{\omega_{b}}
(t)}{\sinh\sqrt{-b}t}\int_{0}^{t}\cosh\sqrt{-b}t~~v^{\prime}(s)ds+2\eta
_{\omega_{b}}(t)\int_{D_{t}}\Vert H_S\Vert d\sigma.
\end{aligned}
\end{equation}

Hence, on now applying Lemma \ref{acot integral cos*v' no minimal}, we have:
\begin{equation}
\begin{aligned}
&  \eta_{\omega_{b}}(t)\int_{\partial D_{t}}\frac{\Vert\nabla^{\perp}
r\Vert^{2}}{\Vert\nabla^{S}r\Vert}d\sigma_{t}\leq\eta_{\omega_{b}}
(t)v^{\prime}(t)-\eta_{\omega_{b}}(t)^{2}v(t)-C\sqrt{-b}\eta_{\omega_{b}
}(t)\frac{v(t)}{\sinh\sqrt{-b}t}+\nonumber\\
&  \eta_{\omega_{b}}(t)\frac{\sqrt{-b}(\cosh\sqrt{-b}t-1)}{\sinh\sqrt{-b}
t}v_{0}+2\eta_{\omega_{b}}(t)\int_{D_{t}}\Vert H_S\Vert d\sigma.\nonumber
\end{aligned}
\end{equation}

On the other hand, as $\frac{\sqrt{-b}(\cosh\sqrt{-b}t-1)}{\sinh\sqrt{-b}t}\leq\eta_{\omega_{b}
}(t)~$we obtain:

\begin{equation}
\begin{aligned}
&  \eta_{\omega_{b}}(t)\int_{\partial D_{t}}\frac{\Vert\nabla^{\perp}
r\Vert^{2}}{\Vert\nabla^{S}r\Vert}d\sigma_{t}\leq\\
&  \eta_{\omega_{b}}(t)v^{\prime}(t)-\eta_{\omega_{b}}(t)^{2}v(t)-C\sqrt
{-b}\eta_{\omega_{b}}(t)\frac{v(t)}{\sinh\sqrt{-b}t}\\&+\eta_{\omega_{b}}
(t)^{2}v_{0}+2\eta_{\omega_{b}}(t)\int_{D_{t}}\Vert H_S\Vert d\sigma
\end{aligned}
\end{equation}

Therefore, by replacing the last inequality in (\ref{la_I})
\begin{equation}\label{des_i_nomin}
\begin{aligned}
I(t)\leq\frac{1}{\eta_{\omega_{b}}(t)}R^{\prime}(t)&+\eta_{\omega_{b}
}(t)v^{\prime}(t)+bv(t)\\&+Cb\frac{v(t)}{\sinh\sqrt{-b}t}+\eta_{\omega_{b}
}(t)^{2}v_{0}+2\eta_{\omega_{b}}(t)\int_{D_{t}}\Vert H_S\Vert d\sigma.
\end{aligned}
\end{equation}
\newline

From inequality (\ref{desig_inicial}) and equality (\ref{TheEquation}) in Remark \ref{remark3.4}:
\begin{equation}
\begin{aligned}
\eta_{\omega_{b}}(t)v^{\prime}(t)&+b~v(t)\leq\int_{D_{t}}(b-K_N)d\sigma
+\frac{1}{2}R(t)\\&-2\int_{D_{t}}\Vert H_S\Vert^{2}d\sigma+I(t)+2\pi\chi
(D_{t})+2\int_{\partial D_t}\frac{\Vert H_S\Vert}{\Vert \nabla^S r\Vert}d\sigma_t
\\& \leq \int_{D_{t}}(b-K_N)d\sigma
+\frac{1}{2}R(t)\\&+I(t)+2\pi\chi
(D_{t})+2\int_{\partial D_t}\frac{\Vert H_S\Vert}{\Vert \nabla^S r\Vert}d\sigma_t
\end{aligned}
\end{equation}

So, on applying (\ref{des_i_nomin}) we obtain:
\begin{equation}\label{char_nomin}
\begin{aligned}
&  -2\pi\chi(D_{t})\leq\int_{D_{t}}(b-K_N)d\sigma+\frac{1}{2}R(t)+\frac
{1}{\eta_{\omega_{b}}(t)}R^{\prime}(t)\\
&  +Cb\frac{v(t)}{\sinh\sqrt{-b}t}+\eta_{\omega_{b}}(t)^{2}v_{0}+2\eta
_{\omega_{b}}(t)\int_{D_{t}}\Vert H_S\Vert d\sigma\\& +2\int_{\partial D_t}\frac{\Vert H_S\Vert}{\Vert \nabla^S r\Vert}d\sigma_t
\end{aligned}
\end{equation}

On the other hand, as
\begin{equation}
\begin{aligned}
&  \int_{S}\left(  \frac{\Vert A^{S}\Vert^{2}}{\eta_{\omega_{b}}(t)}+\Vert
H_S\Vert\right)  d\sigma=\underset{t\rightarrow+\infty}{\lim}\left(
\frac{1}{\eta_{\omega_{b}}(t)}\int_{D_{t}}\Vert A^{S}\Vert^{2}d\sigma
+\int_{D_{t}}\Vert H_S\Vert d\sigma\right)\\
&  =\frac{1}{\sqrt{-b}}\int_{S}\Vert A^{S}\Vert^{2}d\sigma+\int_{S}\Vert
H_S\Vert d\sigma<+\infty
\end{aligned}
\end{equation}

\noindent we have that there exists a monotone increasing (sub)sequence
$\{t_{i}\}_{i=1}^{\infty}$ tending to infinity, (namely, $t_{i} \to\infty$
when $i \to\infty$), such that:

\begin{equation}
\lim_{i\rightarrow\infty}\int_{\partial D_{t_{i}}}\left(  \frac{\Vert
A^{S}\Vert^{2}}{\eta_{\omega_{b}}(t)\Vert\nabla^{S}r\Vert}+\frac{\Vert
H_S\Vert}{\Vert\nabla^{S}r\Vert}\right)  d\sigma_{t_{i}}=0
\end{equation}
So, as both addenda are non-negative, we have:
\begin{equation}
\begin{aligned}
\lim_{i\rightarrow\infty}\frac
{1}{\eta_{\omega_{b}}(t_i)}R^{\prime}(t_i)&=
\lim_{i\rightarrow\infty}\int_{\partial D_{t_{i}}}\frac{\Vert A^{S}\Vert^{2}
}{\eta_{\omega_{b}}(t_{i})\Vert\nabla^{S}r\Vert}d\sigma_{t_{i}}=0\\
\lim_{i\rightarrow\infty}\int_{\partial D_{t_{i}}}\frac{\Vert H_S\Vert}
{\Vert\nabla^{S}r\Vert}d\sigma_{t_{i}}&=0.
\end{aligned}
\end{equation}

Let us consider the exhaustion of $S$ by these extrinsic balls, namely,
$\{D_{t_{i}}\}_{i=1}^{\infty}$. Since $\{D_{t_{i}}\}_{i=1}^{\infty}$ is a
family of connected and precompact open sets which exhaust $S$, then
$\{-\chi(D_{r_{i}})\}_{i=1}^{\infty}$ is monotone non-decreasing. Then,
on replacing $t$ for $t_{i}$ and taking limits when $i\rightarrow\infty$ in
inequality (\ref{char_nomin}), we have that:

\begin{align}
&  2\pi \lim_{i\rightarrow\infty}\inf(\{-\chi(D_{r_{k}})\}_{k=i}^{\infty})\leq
\int_{S}(b-K_N)d\sigma+\frac{1}{2}\int_{S}\Vert A^{S}\Vert^{2}d\sigma\\
& + bC\underset{t\rightarrow+\infty}{\lim}\frac{v(t)}{\cosh\sqrt{-b}t-1}
+2\sqrt{-b}\int_{S}\Vert H_S\Vert
d\sigma-bv_{0},\nonumber
\end{align}
since we have the equality between the limits
\begin{equation}
\lim_{t\to \infty}\frac{v(t)}{\sinh
\sqrt{-b}t}=\lim_{t \to \infty} \frac{v(t)}{\cosh\sqrt{-b}t-1} <\infty \end{equation}

Hence, by applying Theorem \ref{Huber}, $S^{2}$ has finite topology and
\begin{align}
&  -2\pi\chi(S)\leq\int_{S}(b-K_N)d\sigma+\frac{1}{2}\int_{S}\Vert
A^{S}\Vert^{2}d\sigma\\
&  + bC\underset{t\rightarrow+\infty}{\lim}\frac{v(t)}{\cosh\sqrt{-b}t-1}
+2\sqrt{-b}\int_{S}\Vert H_S\Vert
d\sigma-bv_{0}.\nonumber
\end{align}


\end{document}